\documentclass[10pt]{article}
\usepackage{amsmath,amsthm,amssymb,amsfonts,amscd, array}
\usepackage{cite}
\usepackage{tikz}
\usepackage{url}
\usepackage[unicode,colorlinks]{hyperref}
\newtheorem{theorem}{Theorem}[section]
\newtheorem{lemma}[theorem]{Lemma}
\newtheorem{corollary}[theorem]{Corollary}

\newtheorem{remark}[theorem]{Remark}
\newtheorem{construction}[theorem]{Construction}

\begin{document}

\centerline{\bf \Large{On a probabilistic problem on finite semigroups}}
\medskip
\centerline{\bf Attila Nagy and Csaba T\'oth}
\medskip
\centerline{Department of Algebra}
\centerline{Budapest University of Technology and Economics}
\centerline{M\H uegyetem rkp. 3, Budapest, 1111, Hungary}
\centerline{nagyat@math.bme.hu (A. Nagy), tcsaba94@gmail.com (Cs. T\'oth)}

\bigskip

\noindent
\centerline{\bf Abstract}

\medskip

In this paper we deal with the following problem: how does the structure of a finite semigroup $S$ depend on the probability that two elements selected at random from $S$, with replacement, define the same inner right translation of $S$. We solve a subcase of this problem. As the main result of the paper, we show how to construct not necessarily finite medial semigroups in which the index of the kernel of the right regular representation equals two.

\bigskip



\section{Introduction and motivation}\label{introduction}
There are many papers in the mathematical literature which use probabilistic methods to study special algebraic structures \cite{Dixon1, Dixon2, Eberhard, Erdos, Gustafson, Liebeck1, Liebeck2, Liebeck3, NagyToth, Palfy, Pyber}. In \cite{NagyToth}, the following problem is examined. If two elements, $a$ and $b$, are selected at random from a finite semigroup S, with replacement, what is the probability $P_{\theta _S}(S)$ that $a$ and $b$ define the same inner right translations of $S$, i.e., $(a, b)\in \theta _S$, where $\theta _S$ denotes the kernel of the right regular representation of $S$.
It is also investigated how does the structure of a finite semigroup S depend on the
probability $P_{\theta _S}(S)$.
It is shown that, for an arbitrary finite semigroup $S$, $P_{\theta _S}(S)\geq \frac{1}{\mid S/\theta _S\mid}$, where $\mid S/\theta _S\mid$ is the index of $\theta _S$.
Equality is satisfied if and only if each $\theta _S$-class of $S$ contains the same number of elements.
In two cases, a solution is given to the problem of how to construct finite semigroups $S$ satisfying the condition $P_{\theta _S}(S)=\frac{1}{\mid S/\theta _S\mid}$. These two cases are: $S$ is an arbitrary semigroup and the index of $\theta_S$ is $1$;
$S$ is a commutative semigroup and the index of $\theta _S$ is $2$.
In our present paper we answer the problem in that case when $S$ is a finite medial semigroup and the index of $\theta _S$ is $2$. We actually deal with a more general problem.
Our main result is Theorem~\ref{thm1}, in which we show how to construct not necessarily finite medial semigroups $S$ in which the index of $\theta_S$ equals $2$ (not necessarily fulfilling that each $\theta _S$-class of $S$ contains the same number of elements). We construct four medial semigroups and show that a semigroup $S$ is a medial semigroup such that the index of $\theta _S$ is $2$ if and only if $S$ is isomorphic to one of these four semigroups.

\section{Preliminaries}\label{prelim}
By a \emph{semigroup} we mean a multiplicative semigroup, that is, a nonempty set together with an associative multiplication.
A transformation of a semigroup $S$ (acting from the right) is called a \emph{right translation} of $S$, if $(xy)\varrho =x(y\varrho)$ for all $x, y\in S$.
The set of all right translations of a semigroup $S$ is a subsemigroup in the semigroup of all transformations of $S$. For an arbitrary element $a$ of a semigroup $S$, let $\varrho _a$ denote the \emph{inner right translation} $x\mapsto xa$ of $S$. It is known that $\Phi _S :a\mapsto \varrho _a$ is a homomorphism of the semigroup $S$ into the semigroup of all right translations of $S$. The homomorphisms $\Phi _S$ is called the \emph{(right) regular representation} of a semigroup $S$. Let $\theta  _S$ denote the kernel of the right regular representations of a semigroup $S$. It is obvious that $(a, b)\in \theta _S$ for elements $a, b\in S$ if and only if $sa=sb$ for all $s\in S$.

\begin{remark}\label{rem0} \rm If $A$ is a $\theta _S$-class of a semigroup $S$, then $\mid sA\mid =1$ for every $s\in S$, because $sa_1=sa_2$ for every $a_1, a_2\in A$.
\end{remark}

A non-empty subset $I$ of a semigroup $S$ is called an \emph{ideal} of $S$ if $as, sa\in I$ for every $a\in I$ and $s\in S$. If $I$ is an ideal of a semigroup $S$, then the relation $\varrho _I$ on $S$ defined by $(a, b)\in \varrho _I$ if and only if $a=b$ or $a, b\in I$ is a congruence on $S$ which is called the \emph{Rees congruence on $S$ determined by $I$}. The equivalence classes of $S$ mod $\varrho _I$ are $I$ itself and every one-element set $\{ a\}$ with $a\in S\setminus I$. The factor semigroup $S/\varrho_I$ is called the \emph{Rees factor semigroup of $S$ modulo $I$}. We shall write $S/I$ instead of $S/\varrho _I$. We may describe $S/I$ as the result of collapsing $I$ into a single (zero) element, while the elements of $S$ outside of $I$ retain their identity.

A homomorphism $\varphi$ of a semigroup $S$ onto an ideal $I\subseteq S$ is called a \emph{retract homomorphism} if $\varphi$ leaves the elements of $I$ fixed. An ideal $I$ of a semigroup $S$ is called a \emph{retract ideal} if there is a retract homomorphism of $S$ onto $I$. In this case, we say that $S$ is a \emph{retract (ideal) extension of $I$ by the Rees factor semigroup $S/I$}.

An element $e$ of a semigroup is called an \emph{idempotent element} if $e^2=e$. If every element of a semigroup $S$ is idempotent, then $S$ is called a \emph{band}.
A semigroup satisfying the identity $ab=a$ (resp., $ab=b$) is called a \emph{left zero (resp., right zero) semigroup}. It is clear that every left (resp., right) zero semigroup is a band. A commutative band is called a \emph{semilattice}.
A semigroup with a zero element $0$ is called a \emph{zero semigroup} if it satisfies the identity $ab=0$.

For a semigroup $S$, let $\omega _S$ denote the universal relation on $S$. The next lemma is a characterization of semigroups in which $\theta _S=\omega _S$.

\begin{lemma}\label{lem1}(\cite[Theorem 3.2]{NagyToth})
$S$ is a semigroup with $\theta _S=\omega _S$ if and only if $S$ is a retract ideal extension of a left zero semigroup by a zero semigroup.
\end{lemma}

\begin{remark}\label{rem1} \rm If $A$ is a $\theta _S$-class of a semigroup $S$ such that $A$ is a subsemigroup of $S$, then $aa_1=aa_2$ for every $a, a_1, a_2\in A$, and hence $\theta _A=\omega _A$. Then, using Lemma~\ref{lem1}, $A$ is a retract ideal extension of a left zero semigroup by a zero semigroup.
\end{remark}

\begin{remark}\label{rem2} \rm If a semigroup $A$ is a retract ideal extension of a left zero semigroup $E_A$ by a zero semigroup, then $A^2=E_A$ and the set of all idempotent elements of $A$ is $E_A$. If $\varphi _A$ is a retract homomorphism of $A$ onto $E_A$ (acting from the left), then
$a_1a_2=\varphi _A(a_1a_2)=\varphi _A(a_1)\varphi _A(a_2)=\varphi _A(a_1)$ for every $a_1, a_2\in A$.
\end{remark}

A semigroup $S$ is called a \emph{left,  (resp., right) commutative semigroup} if it satisfies the identity $xya=yxa$ (resp., $axy=ayx$).
A semigroup is said to be a \emph{medial semigroup} if it satisfies the identity $axyb=ayxb$. It is clear that every left commutative semigroup and every right commutative semigroup is medial. Left commutative, right commutative and medial semigroups are studied in many papers (see, for example, \cite{Chrislock, Gigon, Halili, Kehayopulu, Nagy:sg-1, Nagy:sg-2, Nagysupplement, Nagymedial, Nagymedleftequ, Nordahl, Strecker:sg-1}) and the books \cite{Nagybook, Petrichbook}. In this paper we also focus on them. The next
 lemma characterizes medial semigroups $S$
 with the help of the factor semigroup $S/\theta _S$.

\begin{lemma}\label{lem2} (\cite[Lemma 3.1]{Nagymedleftequ})
A semigroup $S$ is a medial semigroup if and only if the factor semigroup $S/\theta _S$ is left commutative.
\end{lemma}

\begin{remark}\label{rem3} \rm It is easy to see that a left zero semigroup containing at least two elements is not left commutative. Thus a left commutative semigroup has two elements if and only if it is either a two-element semilattice, a two-element zero semigroup, a two-element group or a two-element right zero semigroup.
\end{remark}

For notions not defined but used in this paper, we refer the reader to books \cite{Clifford1, Nagybook, Petrichbook}.

\section{Results}\label{results}

At the beginning of this section we construct four medial semigroups which play an important role in the proof of our main theorem. All mappings in this sections act from the left.

\begin{construction}\label{constr1}\rm
Let $A$ be a semigroup which is a retract ideal extension of a left zero semigroup $E_A$ by a zero semigroup. Let $\varphi _A$ denote a retract homomorphism of $A$ onto $E_A$. By Remark~\ref{rem2}, the ideal $E_A$ of $A$ is the set of all idempotent elements of $A$, $A^2=E_A$, and $a_1a_2=\varphi _A (a_1)$ for every $a_1, a_2\in A$. Let $B$ be a semigroup such that $A\cap B=\emptyset$. Assume that $B$ is also a retract ideal extension of a left zero semigroup $E_B$ by a zero semigroup. Let $\varphi _B$ denote a retract homomorphism of $B$ onto $E_B$. By Remark~\ref{rem2}, the ideal $E_B$ of $B$ is the set of all idempotent elements of $B$, $B^2=E_B$, and $b_1b_2=\varphi _B(b_1)$ for every $b_1, b_2\in B$.
Let $\alpha$ be a homomorphism of $A$ into $E_B$ and $\beta$ be a homomorphism of $B$ into itself which leaves the elements of $E_B$ fixed. Assume that the equations
\begin{equation}\label{1.1}
\alpha \circ \varphi _A=\alpha,
\end{equation}
\begin{equation}\label{1.2}
\beta \circ \beta =\beta,
\end{equation}
\begin{equation}\label{1.3}
\varphi _B \circ \beta =\varphi _B
\end{equation}
are satisfied.
Since $\alpha$ maps $A$ into $E_B$, and the mappings $\varphi _B$ and $\beta$ leave the elements of $E_B$ fixed, we have
\begin{equation}\label{1.4}
\varphi _B \circ \alpha =\alpha =\beta \circ \alpha,
\end{equation}
\begin{equation}\label{1.5}
\beta \circ \varphi _B=\varphi _B.
\end{equation}
Let $S=A\cup B$. We define an operation $\cdot$ on $S$. For every $x, y\in S$, let
\[x\cdot y=\left\{ \begin{array}{ll}
\varphi_A(x) & \textrm{if $x, y\in A$,}\\
\alpha (x) & \textrm{if $x\in A$, $y\in B$,}\\
\beta (x) & \textrm{if $x\in B$, $y\in A$,}\\
\varphi _B(x) & \textrm{if $x, y\in B$.}\end{array}\right. \]
We show that the operation $\cdot$ is associative. Let $a\in A$ and $x, y\in S$ be arbitrary elements. Using equations (\ref{1.1}) and  (\ref{1.4}), we have
\[(a\cdot x)\cdot y=\left\{ \begin{array}{ll}
\varphi _A(a) & \textrm{if $x, y\in A$},\\
\alpha (a) & \textrm{otherwise.}\end{array}\right. \] By the definition of the operation $\cdot$, we have
\[a\cdot (x\cdot y)=\left\{ \begin{array}{ll}
\varphi _A(a) & \textrm{if $x, y\in A$},\\
\alpha (a) & \textrm{otherwise.}\end{array}\right. \]
Thus \[(a\cdot x)\cdot y=a\cdot (x\cdot y).\]
Let $b\in B$ and $x, y\in S$ be arbitrary elements. Using (\ref{1.2}), (\ref{1.3}) and (\ref{1.5}), we have
\[(b\cdot x)\cdot y=\left\{\begin{array}{ll} \beta (b) & \textrm{if $x, y\in A$},\\
\varphi _B(b) & \textrm{otherwise.}\end{array}\right. \] By the definition of the operation $\cdot$, we have
\[b\cdot (x\cdot y)=\left\{\begin{array}{ll} \beta (b), & \textrm{if $x, y\in A$},\\
\varphi _B(b), & \textrm{otherwise.}\end{array}\right. \] Thus \[(b\cdot x)\cdot y=b\cdot (x\cdot y).\]
Consequently the operation $\cdot$ is associative, and hence the algebraic structure $(S; \cdot )$ is a semigroup.
It is clear that $(S; \cdot )$ is a right commutative semigroup, and hence a medial semigroup.
It follows from the definition of the operation $\cdot$ that $(a_1, a_2)\in \theta _S$ and $(b_1, b_2)\in \theta _S$ for every $a_1, a_2\in A$ and $b_1, b_2\in B$.
For every $a\in A, b\in B$, we have $a\cdot a\in A$ and $a\cdot b\in B$, and hence $(a, b)\notin \theta _S$. Thus the $\theta _S$-classes of $S$ are $A$ and $B$.
Since $A^2\subseteq A$, $B^2\subseteq B$, $A\cdot B, B\cdot A\subseteq B$, the factor semigroup $S/\theta _S$ is a two-element semilattice.
\end{construction}

To illustrate Construction~\ref{constr1}, consider the following example. Let $A=\{ a, 0\}$ be a two-element zero semigroup ($0$ is the zero of $A$, and so $E_A=\{ 0\}$) and $B=\{ e, f\}$ be a two-element left zero semigroup (and hence $\varphi _B$ is the identity mapping of $B$). Let $\alpha$ be a mapping of $A$ into $B$ such that $\alpha (x)=e$ for every $x\in A$. It is clear that $\alpha$ is a homomorphism. Let $\beta$ be the identity mapping of $B$ (that is, $\beta =\varphi _B$). Then $\beta$ leaves the elements of $E_B$ fixed. The equations (\ref{1.1}), (\ref{1.2}), and (\ref{1.3}) are satisfied. The Cayley-table of the semigroup $(S; \cdot )$ is Table~\ref{Fig1}.

\begin{table}[htbp]
\begin{center}
\begin{tabular}{l|l l l l }
$\cdot$&$a$&$0$&$e$&$f$\\ \hline
$a$&$0$&$0$&$e$&$e$\\
$0$&$0$&$0$&$e$&$e$\\
$e$&$e$&$e$&$e$&$e$\\
$f$&$f$&$f$&$f$&$f$\\
\end{tabular}
\medskip
\caption{}\label{Fig1}
\end{center}
\end{table}

\begin{construction}\label{constr2}\rm
Let $B$ be a semigroup which is a retract ideal extension of a left zero semigroup $E_B$ by a zero semigroup. Let $\varphi _B$ denote a retract homomorphism of $B$ onto $E_B$. By Remark~\ref{rem2}, the ideal $E_B$ of $B$ is the set of all idempotent elements of $B$, $B^2=E_B$, and $b_1b_2=\varphi _B(b_1)$ for every $b_1, b_2\in B$. Let $A$ be a nonempty set such that $A\cap B=\emptyset$.
Let $\alpha$ be a mapping of $A$ into $B$, $\beta$ be a mapping of $A$ into $E_B$ and $\gamma$ be a homomorphism of $B$ into itself which leaves the elements of $E_B$ fixed. Assume $\alpha \neq \beta$ or $\gamma \neq \varphi _B$, and
\begin{equation}\label{2.1}
\varphi _B\circ \alpha =\beta =\gamma \circ \alpha,
\end{equation}
\begin{equation}\label{2.2}
\varphi _B\circ \gamma =\varphi _B= \gamma \circ \gamma.
\end{equation}
Since $\beta (A)\subseteq E_B$ and $\gamma$ leaves the element of $E_B$ fixed, we have
\begin{equation}\label{2.3}
\varphi _B\circ \beta =\beta =\gamma \circ \beta,
\end{equation}
\begin{equation}\label{2.4}
\gamma \circ \varphi _B=\varphi _B.
\end{equation}
Let $S=A\cup B$. We define an operation $\bullet$ on $S$. For every $x, y\in S$, let
\[x\bullet y=\left\{ \begin{array}{ll}
\alpha (x) & \textrm{if $x, y\in A$},\\
\beta(x) & \textrm{if $x\in A$ and $y\in B$},\\
\gamma (x) & \textrm{if $x\in B$, $y\in A$},\\
\varphi _B(x) & \textrm{if $x, y\in B$}.
\end{array} \right. \]
We show that the operation $\bullet$ is associative. Let $a\in A$ and $x, y\in S$ be arbitrary elements. Using (\ref{2.1}) and (\ref{2.3}),
we have
\[(a\bullet x)\bullet y=\beta (a).\]
By the definition of the operation $\bullet$, we have
\[a\bullet (x\bullet y)=\beta (a).\] Thus \[(a\bullet x)\bullet y=a\bullet (x\bullet y).\]

Let $b\in B$ and $x, y\in S$ be an arbitrary element. Using (\ref{2.2}), (\ref{2.4}) and the fact that $\varphi_B$ is a retract homomorphism,
we have \[(b\bullet x)\bullet y=\varphi _B(b).\]
By the definition of the operation $\bullet$, we have \[b\bullet (x\bullet y)=\varphi _B(b).\]
Thus \[(b\bullet x)\bullet y=b\bullet (x\bullet y).\]
Consequently the operation $\bullet$ is associative, and hence the algebraic structure $(S; \bullet )$ is a semigroup.
It is clear that $(S; \bullet )$ is a right commutative semigroup, and hence a medial semigroup.
It follows from the definition of the operation $\bullet$ that $(a_1, a_2)\in \theta _S$ and $(b_1, b_2)\in \theta _S$ for every $a_1, a_2\in A$ and $b_1, b_2\in B$.
Let $a\in A$ and $b\in B$ be arbitrary elements. Assume $(a, b)\in \theta_S$. Then, for every $a'\in A$ and every $b'\in B$, we have
$\alpha (a')=a'\bullet a=a'\bullet b=\beta (a')$ and
$\gamma (b')=b'\bullet a=b'\bullet b=\varphi _B(b')$. Thus $\alpha =\beta$ and $\gamma =\varphi _B$. This contradict the assumption that $\alpha \neq \beta$ or $\gamma \neq \varphi _B$.  Consequently the $\theta _S$-classes of $S$ are $A$ and $B$.
Since $A^2\subseteq B$, $B^2\subseteq B$, $A\bullet B, B\bullet A\subseteq B$, the factor semigroup $S/\theta _S$ is a two-element zero semigroup.
\end{construction}

To illustrate Construction~\ref{constr2}, consider the following example. Let $B=\{b, 0\}$ be a two-element zero semigroup ($0$ is the zero of $B$) and $A=\{ a\}$ be a singleton. Let $\alpha (a)=b$, $\beta (a)=0$, and $\gamma =\varphi _B$. Then $\beta (A)=\{ 0\} =E_B$, and $\gamma$ leaves the elements of $E_B$ fixed. Since $b\neq 0$, we have $\alpha \neq \beta$. Equations (\ref{2.1}) and (\ref{2.2}) are also satisfied. The Cayley-table of the semigroup $(S; \bullet )$ is Table~\ref{Fig2}.

\begin{table}[htbp]
\begin{center}
\begin{tabular}{l|l l l }
$\bullet$&$a$&$b$&$0$\\ \hline
$a$&$b$&$0$&$0$\\
$b$&$0$&$0$&$0$\\
$0$&$0$&$0$&$0$\\
\end{tabular}
\medskip
\caption{}\label{Fig2}
\end{center}
\end{table}


\begin{construction}\label{constr3}\rm
Let $A$ be a semigroup which is a retract ideal extension of a left zero semigroup $E_A$ by a zero semigroup. Let $\varphi _A$ denote a retract homomorphism of $A$ onto $E_A$. By Remark~\ref{rem2}, the ideal $E_A$ of $A$ is the set of all idempotent elements of $A$, $A^2=E_A$, and $a_1a_2=\varphi _A (a_1)$ for every $a_1, a_2\in A$. Let $B$ be a nonempty set. Let $\alpha$ be a mapping of $A$ into $B$ and $\gamma$ be a mapping of $B$ into $E_A$ such that the equations
\begin{equation}\label{3.1}
\alpha =\alpha \circ \varphi _A,
\end{equation}
\begin{equation}\label{3.2}
\gamma \circ \alpha =\varphi _A,
\end{equation}
are satisfied. Let
\begin{equation}\label{beta}
\beta=\alpha \circ \gamma.
\end{equation}
From the equations (\ref{3.1}) and (\ref{3.2}), it follows that
\begin{equation}\label{betanegyzet}
\beta \circ \beta =\beta,
\end{equation}
\begin{equation}\label{3.3}
\beta \circ \alpha=\alpha \circ \gamma \circ \alpha =\alpha \circ \varphi _A=\alpha.
\end{equation}
Using the fact that $\gamma (b)\in E_A$ and $\varphi _A$ leaves the elements of $E_A$ fixed, the equations (\ref{beta}) and (\ref{3.2}) imply
\begin{equation}\label{3.4}
\gamma \circ \beta =\gamma \circ \alpha \circ \gamma =\varphi _A \circ \gamma =\gamma.
\end{equation}
Let $S=A\cup B$. We define an operation $*$ on $S$. For every $x, y\in S$, let
\[x* y=\left\{ \begin{array}{ll}
\varphi _A(x) & \textrm{if $x, y\in A$},\\
\alpha (x) & \textrm{if $x\in A$ and $y\in B$},\\
\beta(x) & \textrm{if $x\in B, y\in A$},\\
\gamma (x)  & \textrm{if $x, y\in B$}.\end{array} \right. \]
We show that the operation $*$ is associative. Let $a\in A$ and $x, y\in S$ be an arbitrary elements. Using the equations (\ref{3.1}) and (\ref{3.3}),
we have
\[(a*x)*y=\left\{ \begin{array}{ll}
\varphi _A(a), & \textrm{if $x, y\in A$ or $x, y\in B$},\\
\alpha (a), & \textrm{otherwise.}\end{array}\right. \]
By the definition of the operation $*$, we have
\[a*(x*y)=\left\{ \begin{array}{ll}
\varphi _A(a), & \textrm{if $x, y\in A$ or $x, y\in B$},\\
\alpha (a), & \textrm{otherwise.}\end{array}\right. \] Thus \[(a*x)*y=a*(x*y).\]
Let $b\in B$ and $x, y\in S$ be an arbitrary elements. Using the equations (\ref{beta}) and (\ref{3.4}),
we have
\[(b*x)*y=\left\{ \begin{array}{ll}
\beta (b), & \textrm{if $x, y\in A$ or $x, y\in B$},\\
\gamma (b), & \textrm{otherwise.}\end{array}\right. \]
By the definition of the operation $*$, we have
\[b*(x*y)=\left\{ \begin{array}{ll}
\beta (b), & \textrm{if $x, y\in A$ or $x, y\in B$},\\
\gamma (b), & \textrm{otherwise.}\end{array}\right. \] Thus \[(b*x)*y=b*(x*y).\]
Consequently the operation $*$ is associative, and hence the algebraic structure $(S; * )$ is a semigroup.
It is clear that $(S; *)$ is a right commutative semigroup, and hence a medial semigroup.
It follows from the definition of the operation $*$ that $(a_1, a_2)\in \theta _S$ and $(b_1, b_2)\in \theta _S$ for every $a_1, a_2\in A$ and $b_1, b_2\in B$. Let $a\in A$ and $b\in B$ be arbitrary elements. Assume $(a, b)\in \theta_S$. Then, for every $a'\in A$, we have
\[A\ni a'*a=a'*b\in B\] which is a contradiction. Thus the $\theta _S$-classes of $S$ are $A$ and $B$.
Since $A^2\subseteq A$, $A*B, B*A\subseteq B$ and $B^2\subseteq A$, the factor semigroup $S/\theta _S$ is a two-element group.
\end{construction}

To illustrate Construction~\ref{constr3}, consider the following example. Let $A=\{ e\}$ be a one-element semigroup and $B=\{x, y\}$ be a two-element set. Let $\alpha$ be the mapping of $A$ onto the subset $\{ x\}$. Let $\gamma$ be the only possible mapping of $B$ onto $A$. We note that $\beta (b)=x$ for every $b\in B$.
It is easy to see that the above conditions for $\alpha$ and $\gamma$ are satisfied. The Cayley-table of the semigroup $(S; *)$ is Table~\ref{Fig3}.
\begin{table}[htbp]
\begin{center}
\begin{tabular}{l|l l l }
$*$&$e$&$x$&$y$\\ \hline
$e$&$e$&$x$&$x$\\
$x$&$x$&$e$&$e$\\
$y$&$x$&$e$&$e$\\
\end{tabular}
\medskip
\caption{}\label{Fig3}
\end{center}
\end{table}

\begin{construction}\label{constr4}\rm
Let $A$ be a semigroup which is a retract ideal extension of a left zero semigroup $E_A$ by a zero semigroup. Let $B$ be a semigroup such that $A\cap B=\emptyset$, and $B$ is also a retract ideal extension of a left zero semigroup $E_B$ by a zero semigroup. Let $\varphi _A$ and $\varphi _B$ denote a retract homomorphism of $A$ onto $E_A$ and $B$ onto $E_B$, respectively.
By Remark~\ref{rem2}, the ideal $E_A$ of $A$ is the set of all idempotent elements of $A$, $A^2=E_A$, and $a_1a_2=\varphi _A(a_1)$ for every $a_1, a_2\in A$. Similarly, the ideal $E_B$ of $B$ is the set of all idempotent elements of $B$, $B^2=E_B$, and $b_1b_2=\varphi _B(b_1)$ for every $b_1, b_2\in B$.
Let $\alpha$ be a homomorphism of $A$ into $E_B$ and $\beta$ be a homomorphism of $B$ into $E_A$. Assume that the equations
\begin{equation}\label{4.1}
\alpha \circ \varphi _A=\alpha =\varphi _B\circ \alpha,
\end{equation}
\begin{equation}\label{4.2}
\beta \circ \varphi _B=\beta =\varphi _A\circ \beta,
\end{equation}
\begin{equation}\label{4.3}
\alpha \circ \beta =\varphi _B,
\end{equation}
\begin{equation}\label{4.4}
\beta \circ \alpha =\varphi _A.
\end{equation}
are satisfied.
Let $S=A\cup B$. We define an operation $\star$ on $S$. For every $x, y\in S$, let
\[x\star y=\left\{ \begin{array}{ll}
\varphi _A (x) & \textrm{if $x, y\in A$},\\
\alpha (x) & \textrm{if $x\in A$, $y\in B$},\\
\beta (x) & \textrm{if $x\in B$, $y\in A$},\\
\varphi _B(x) & \textrm{if $x, y\in B$.}\end{array}\right. \]
We show that the operation $\star$ is associative. Let $a\in A$ and $x, y\in S$ be arbitrary elements. Using the equations (\ref{4.1}) and  (\ref{4.4}), we have
\[(a\star x)\star y=\left\{ \begin{array}{ll}
\varphi _A(a) & \textrm{if $y\in A$},\\
\alpha (a) & \textrm{otherwise.}\end{array}\right. \] By the definition of the operation $\cdot$, we have
\[a\star (x\star y)=\left\{ \begin{array}{ll}
\varphi _A(a) & \textrm{if $y\in A$},\\
\alpha (a) & \textrm{otherwise.}\end{array}\right. \] Thus \[(a\star x)\star y=a\star (x\star y).\]
Let $b\in B$ and $x, y\in S$ be arbitrary elements. Using (\ref{4.2}) and (\ref{4.3}), we have
\[(b\star x)\star y=\left\{ \begin{array}{ll}
\beta (b) & \textrm{if $y\in A$},\\
\varphi _B(b) & \textrm{if $y\in B$}.\end{array}\right. \] By the definition of the operation $\cdot$, we have
\[b\star (x\star y)=\left\{ \begin{array}{ll}
\beta (b) & \textrm{if $y\in A$},\\
\varphi _B(b) & \textrm{if $y\in B$}.\end{array}\right. \]
Thus \[(b\star x)\star y=b\star (x\star y).\]
Consequently the operation $\star$ is associative. It is clear that the semigroup $(S; \star )$ is a medial semigroup.
It follows from the definition of the operation $\star$ that $(a_1, a_2)\in \theta _S$ and $(b_1, b_2)\in \theta _S$ for every $a_1, a_2\in A$ and $b_1, b_2\in B$.
For every $a\in A, b\in B$, we have $a\star a\in A$ and $a\star b\in B$, and hence $(a, b)\notin \theta _S$. Thus the $\theta _S$-classes of $S$ are $A$ and $B$.
Since $A^2\subseteq A$, $B^2\subseteq B$, $A\star B\subseteq B$, $B\star A\subseteq A$, the factor semigroup $S/\theta _S$ is a two-element right zero semigroup.
\end{construction}

To illustrate Construction~\ref{constr4}, consider the following example. Let $A=\{ e, f\}$ and $B=\{ g, h\}$ be disjoint two-element left zero semigroups. Let $\alpha$ be a mapping of $A$ into $B$ such that $\alpha (e)=g$ and $\alpha (f)=h$. Let $\beta$ be the mapping of $B$ into $A$ such that $\beta (g)=e$ and $\beta (h)=f$ (that is, $\beta$ and $\alpha$ are inverses of each other). Since $A$ and $B$ are left zero semigroups, both of $\alpha$ and $\beta$ are homomorphisms. Since $E_A=A$ and $E_B=B$, $\alpha$ maps $A$ into $E_B$ and $\beta$ maps $B$ into $E_A$. In this example, $\varphi _A$ and $\varphi _B$ are the identities on the sets $A$ and $B$, respectively, which fact together with the note that the mapings $\alpha$ and $\beta$ are mutually inverse, readily yields the validity of the equations (\ref{4.1})--(\ref{4.4}). The Cayley-table of the semigroup $(S; \star )$ is Table~\ref{Fig4}.

\begin{table}[htbp]
\begin{center}
\begin{tabular}{l|l l l l }
$\star$&$e$&$f$&$g$&$h$\\ \hline
$e$&$e$&$e$&$g$&$g$\\
$f$&$f$&$f$&$h$&$h$\\
$g$&$e$&$e$&$g$&$g$\\
$h$&$f$&$f$&$h$&$h$\\
\end{tabular}
\medskip
\caption{}\label{Fig4}
\end{center}
\end{table}

\begin{theorem}\label{thm1} A semigroup $S$ is a medial semigroup such that the index of $\theta _S$ is $2$ if and only if $S$ is isomorphic to one of the semigroups defined in Construction~\ref{constr1}, Construction~\ref{constr2}, Construction~\ref{constr3}, and Construction~\ref{constr4}.
\end{theorem}

\begin{proof} Semigroups $S$ defined in Construction~\ref{constr1}, Construction~\ref{constr2}, Construction~\ref{constr3}, and Construction~\ref{constr4} are medial semigroups such that the index of $\theta _S$ is $2$.

Conversely, let $S$ be a medial semigroup such that the index of $\theta _S$ is $2$. By Lemma~\ref{lem2} and Remark~\ref{rem3}, $S/\theta _S$ is either a two-element semilattice, a two-element zero semigroup, a two-element group or a two-element right zero semigroup.

First consider the case when $S/\theta _S$ is a two-element semilattice.
Let $A$ and $B$ denote the $\theta _S$-classes of $S$. Then $A$ and $B$ are subsemigroups of $S$ such that one of them is an ideal of $S$. Assume that $B$ is an ideal of $S$. By Remark~\ref{rem1}, $\theta _A=\omega _A$ and $\theta _B=\omega _B$. Thus, by Lemma~\ref{lem1} and Remark~\ref{rem2}, $A$ is a retract ideal extension of the set $E_A$ of all idempotent elements of $A$ by a zero semigroup, and $B$ is a retract ideal extension of the set of all idempotent elements $E_B$ of $B$ by a zero semigroup. Let $\varphi _A$ denote a retract homomorphism of $A$ onto $E_A$ and $\varphi _B$ denote a retract homomorphism of $B$ onto $E_B$. By Remark~\ref{rem2}, $a_1a_2=\varphi _A(a_1)$ for every $a_1, a_2\in A$, and $b_1b_2=\varphi _B(b_1)$ for every $b_1, b_2\in B$. By Remark~\ref{rem0}, $\mid aB\mid =1$ for every $a\in A$. Let $\alpha$ be a mapping of $A$ into $B$ defined by
\[\alpha (a)=aB, \quad a\in A.\] For every $a_1, a_2\in A$, we have
\[\alpha (a_1a_2)=a_1a_2B=a_1Ba_2B=\alpha (a_1)\alpha (a_2),\] because $B$ is a $\theta _S$-class, $a_2B, Ba_2B\subseteq B$, and hence $a_1a_2B=a_1Ba_2B$. Thus $\alpha$ is a homomorphism.
For every $a\in A$, we have \[\alpha (a)\alpha (a)=aBaB=aB=\alpha (a),\] because $BaB\subseteq B$, and hence $aBaB=aB$.
Thus $\alpha$ maps $A$ into the set $E_B$ of all idempotent elements of $B$.

By Remark~\ref{rem0}, $\mid bA\mid=1$ for every $b\in B$. Let $\beta$ be a mapping of $B$ into itself defined by
\[\beta (b)=bA, \quad b\in B.\] Let $e$ be an idempotent element of $A$. Since $S$ is a medial semigroup, we have, for every $b_1, b_2\in B$,
\[ \beta (b_1b_2)=b_1b_2A=b_1b_2e=b_1b_2ee=b_1eb_2e=(b_1A)(b_2A)=\beta (b_1)\beta (b_2),\] that is, $\beta$ is a homomorphism. If $f\in E_B$, then
\[\beta(f)=fA=ffA=ff=f,\] because $B$ is a $\theta _S$-class of $S$ and $fA\subseteq B$. Thus $\beta$ leaves the elements of $E_B$ fixed.
For every $a, a^*\in A$ and $b\in B$, we have
\[(\alpha \circ \varphi _A)(a)=\alpha(\varphi _A(a))=\alpha(aa^*)=(aa^*)b=a(a^*b)=\alpha (a),\] and hence (\ref{1.1}) is satisfied.
For every $a, a^*\in A$ and $b\in B$, we have
\[(\beta \circ \beta )(b)=\beta (\beta (b))=\beta (b)a=(ba^*)a=b(a^*a)=\beta (b).\] Thus (\ref{1.2}) is satisfied.
For every $b, b^*\in B$ and $a\in A$, we have
\[(\varphi _B\circ \beta)(b)=\varphi _B(\beta (b))=\beta (b)b^*=(ba)b^*=b(ab^*)=\varphi _B(b),\] and hence (\ref{1.3}) is satisfied. We can consider the semigroup $(S; \cdot )$ defined in Construction~\ref{constr1}. It is clear that the semigroup $S$ is isomorphic to the semigroup $(S; \cdot )$.

In the second step of the proof, suppose that $S/\theta _S$ is a two-element zero semigroup.
Let $A$ and $B$ denote the $\theta_S$-classes of $S$. We can suppose that $B$ is an ideal of $S$. Then $A^2\subseteq B$. By Remark~\ref{rem1}, $\theta _B=\omega _B$. Then, by Lemma~\ref{lem1} and Remark~\ref{rem2}, $B$ is a retract ideal extension of a left zero semigroup $E_B$ by a zero semigroup, where $E_B$ is the set of all idempotent elements of $B$. Moreover, for arbitrary $b_1, b_2\in B$, we have $b_1b_2=\varphi _B(b_1)$. By Remark~\ref{rem0}, $\mid aA\mid =\mid aB\mid =\mid bA\mid =1$ for every $a\in A$ and $b\in B$. For arbitrary $a\in A$ and $b\in B$, let
\[\alpha (a)=aA,\quad \beta (a)=aB,\quad \gamma (b)=bA.\]
Since $bab\in B$ for every $a\in A$ and $b\in B$ (and hence $(bab, b)\in\theta _S$), we have \[(ab)^2=a(bab)=ab.\] Thus $\beta (a)\in E_B$ for every $a\in A$. For every $b_1, b_2\in B$, we have \[\gamma (b_1b_2)=(b_1b_2)A=b_1(b_2A)=b_1(Ab_2A)=(b_1A)(b_2A)=\gamma(b_1)\gamma(b_2),\] because $b_2A, Ab_2A\subseteq B$ and, by Remark~\ref{rem0}, $\mid b_1B\mid =1$. Thus $\gamma$ is a homomorphism.
For every $f\in E_B$ and $a\in A$, we have $(fa, f)\in \theta _S$, and hence \[fa=f(fa)=ff=f.\] Thus \[\gamma (f)=fA=f,\] that is, $\gamma$ leaves the elements of $E_B$ fixed.

Since $AB, AA\subseteq B$, we have \[a(AB)=aB=a(AA)\] for every $a\in A$. Then \[\varphi _B\circ \alpha =\beta =\gamma \circ \alpha,\] that is, (\ref{2.1}) is satisfied. Using also the inclusions $AB, AA\subseteq B$, we get \[bAB=bB=b(AA)\] for every $b\in B$. Then
\[\varphi _B\circ \gamma =\varphi _B= \gamma \circ \gamma,\] that is, (\ref{2.2}) is satisfied.
Thus we can consider the semigroup $(S; \bullet )$ defined in Construction~\ref{constr2}. It is clear that $S$ is isomorphic to $(S; \bullet )$.

In the third step of the proof, suppose that $S/\theta _S$ is a two-element group.
Let $A$ and $B$ denote the $\theta_S$ -classes of $S$. We can suppose that $A$ is the identity element of $S/\theta _S$. Then $A$ is a subsemigroup of $S$, $AB, BA\subseteq B$, and $B^2\subseteq A$. By Remark~\ref{rem1}, $\theta _ A=\omega _A$. Then, by Lemma~\ref{lem1} and Remark~\ref{rem2}, $A$ is a retract extension of a left zero semigroup $E_A$ by a zero semigroup where $E_A$ is the set of all idempotent elements of $A$. Moreover, for arbitrary $a_1, a_2\in A$, we have $a_1a_2=\varphi _A(a_1)$. By Remark~\ref{rem0}, $\mid aB\mid =\mid bA\mid =\mid bB\mid =1$ for every $a\in A$ and $b\in B$. For arbitrary $a\in A$ and $b\in B$, let
\[\alpha(a)=aB, \quad \beta (b)=bA, \quad \gamma(b)=bB.\]
Since $b_2b_1b_2\in B$ for every $b_1, b_2\in B$ (and so $(b_2b_1b_2, b_2)\in \theta _S$), we have \[(b_1b_2)^2=b_1(b_2b_1b_2)=b_1b_2.\] Thus $\gamma$ maps $B$ into $E_A$.

Let $b\in B$ be an arbitrary element. Since $B^2\subseteq A$, we have
\[\beta (b)=bA=bB^2=(bB)B=\gamma (b)B=\alpha (\gamma (b))=(\alpha \circ \gamma)(b).\] Thus
\[\beta =\alpha \circ \gamma,\] and hence (\ref{beta}) is satisfied.
For arbitrary $a\in A$, we have
\[(\alpha \circ \varphi _A)(a)=\alpha (\varphi _A(a))=\varphi _A(a)B=(aa)B=a(aB)=\alpha (a)\] and
\[(\gamma \circ \alpha)(a)=\gamma (\alpha (a))=\alpha (a)B=(aB)B=a(BB)=\varphi _A(a).\] Thus conditions (\ref{3.1}) and (\ref{3.2}) are satisfied.
Hence we can consider the semigroup $(S; * )$ defined in Construction~\ref{constr3}. It is clear that $S$ is isomorphic to $(S; * )$.

In the fourth step of the proof, suppose that $S/\theta _S$ is a two-element right zero semigroup.
Let $A$ and $B$ denote the $\theta _S$-classes of $S$. Then $A$ and $B$ are subsemigroups of $S$ such that $AB\subseteq B$ and $BA\subseteq A$.
By Remark~\ref{rem1}, $\theta _A=\omega _A$ and $\theta _B=\omega _B$. Thus, by Lemma~\ref{lem1} and Remark~\ref{rem2}, $A$ is a retract ideal extension of a left zero semigroup $E_A$ by a zero semigroup, where $E_A$ is the set of all idempotent elements of $A$. Moreover, for arbitrary $a_1, a_2\in A$, we have $a_1a_2=\varphi _A(a_1)$. Similarly, $B$ is a retract ideal extension of a left zero semigroup $E_B$ by a zero semigroup, where $E_B$ is the set of all idempotent elements of $B$. Moreover, for arbitrary $b_1, b_2\in B$, we have $b_1b_2=\varphi _B(b_1)$.

For every $a\in A$ and $b\in B$, $\mid aB\mid =\mid bA\mid =1$ by Remark~\ref{rem0}. Let $\alpha$ be a mapping of $A$ into $B$ defined by \[\alpha (a)=aB,\quad a\in A,\] and $\beta$ be a mapping of $B$ into $A$ defined by \[\beta (b)=bA, \quad b\in B.\]

For every $a_1, a_2\in A$, we have $a_2B, Ba_2B\subseteq B$, and hence $a_1a_2B=a_1Ba_2B$, because $\mid a_1B\mid =1$ by Remark~\ref{rem0}. Then
\[\alpha (a_1a_2)=a_1a_2B=a_1Ba_2B=\alpha (a_1)\alpha (a_2).\] Thus $\alpha$ is a homomorphism.
For every $a\in A$, we have $BaB\subseteq B$, and hence $a(BaB)=aB$, because $\mid aB\mid =1$ by Remark~\ref{rem0}. Then \[(\alpha (a))^2=\alpha (a)\alpha (a)=aBaB=aB=\alpha (a),\] and hence $\alpha$ maps $A$ into $E_B$.
It can be similarly proved that $\beta$ is a homomorphism which maps $B$ into $E_A$.

For arbitrary $a\in A$ and $b\in B$,
\[(\alpha \circ \varphi _A)(a)=\alpha (\varphi _A(a))=\varphi _A(a)b=(aa)b=a(ab)=\alpha (a)\] and
\[(\varphi _B\circ \alpha)(a)=\varphi _B(\alpha (a))=\alpha (a)b=(ab)b=a(bb)=\alpha (a).\] Thus (\ref{4.1}) is satisfied.
We can proof in a similar way that (\ref{4.2}) is satisfied.
For arbitrary $a\in A$ and $b\in B$, \[(\alpha \circ \beta )(b)=\alpha (\beta (b))=\beta (b)b=(ba)b=b(ab)=\varphi _B(b).\] Thus (\ref{4.3}) is satisfied. (\ref{4.4}) can be similarly proved.
Consider the semigroup $(S; \star )$ defined as in Construction~\ref{constr4}. It is clear that $S$ is isomorphic to $(S; \star )$.
\end{proof}

By \cite[Dual of Lemma 3.2]{Nagymedleftequ}, a semigroup $S$ is right commutative if and only if the factor semigroup $S/\theta _S$ is commutative. This result and the proof of Theorem~\ref{thm1} together imply the following corollary.

\begin{corollary}\label{rightcomm} A semigroup $S$ is a right commutative semigroup such that the index of $\theta _S$ is $2$ if and only if $S$ is isomorphic to one of the semigroups defined in Construction~\ref{constr1}, Construction~\ref{constr2}, and Construction~\ref{constr3}.
\end{corollary}

By \cite{NagyToth}, in a finite semigroup $S$ both of the conditions that the index of $\theta _S$ is $m$ and $P_{\theta _S}(S)=\frac{1}{m}$ are satisfied if and only if  each $\theta_S$-class contains the same number of elements. Thus the following result is a consequence of Theorem~\ref{thm1}.

\begin{theorem}\label{thm2}  $S$ is a finite medial semigroup such that the index of $\theta _S$ is $2$ and  $P_{\theta _S}(S)=\frac{1}{2}$ if and only if $S$ is isomorphic to one of the semigroups defined in Construction~\ref{constr1}, Construction~\ref{constr2}, Construction~\ref{constr3}, and Construction~\ref{constr4} satisfying the condition $\mid A\mid =\mid B\mid <\infty$ in each of the four cases.
\end{theorem}

Corollary~\ref{rightcomm} and Theorem~\ref{thm2} imply the following corollary.

\begin{corollary} $S$ is a finite right commutative semigroup such that the index of $\theta _S$ equals $2$ and $P_{\theta _S}(S)=\frac{1}{2}$ if and only if $S$ is isomorphic to one of the semigroups defined in Construction~\ref{constr1}, Construction~\ref{constr2}, and Construction~\ref{constr3} satisfying the condition $\mid A\mid =\mid B\mid <\infty$ in each of the three cases.
\end{corollary}

\bibliographystyle{line}

\end{document}